\documentclass[11pt,oneside, a4paper]{amsart}

\usepackage[english]{babel} 
\usepackage[T1]{fontenc} 
\usepackage{amsmath}
\usepackage{amssymb} 
\usepackage{float}
\usepackage{amsfonts}
\usepackage{mathtools}
\usepackage[utf8]{inputenc}
\usepackage[mathcal]{eucal} 
\usepackage{enumerate}
\usepackage{amsthm} 
\usepackage{hyperref}
\usepackage[totalwidth=14cm,totalheight=20cm]{geometry}
\usepackage{multicol}

\newtheorem{defi}{Definition}[section] 
\newtheorem{teo}[defi]{Theorem}
 
\newtheorem{lemma}[defi]{Lemma}
\newtheorem{prop}[defi]{Proposition}
\newtheorem{oss}[defi]{Remark}

\newcommand{\Pp}{\mathbb{P}}
\newcommand{\R}{\mathbb{R}}

\newcommand{\h}{\mathbb{H}}
\newcommand{\N}{\mathbb{N}} 
\newcommand{\K}{\kappa}

\newcommand{\grad}{\mathrm{grad}}
\newcommand{\trace}{\mathrm{trace}}
\newcommand{\dive}{\mathrm{div}}
\newcommand{\Hopf}{\mathrm{Hopf}}

\DeclareMathAlphabet{\mathpzc}{OT1}{pzc}{m}{it}

\title[CMC foliation of domains of dependence]{Constant mean curvature foliation \\ of domains of dependence in $AdS_{3}$} 
\author{Andrea Tamburelli}
\date{\today}
\thanks{}

\begin{document}

\begin{abstract}  
We prove that, given an acausal curve $\Gamma$ in the boundary at infinity of $AdS_{3}$ which is the graph of a quasi-symmetric homeomorphism $\phi$, there exists a unique foliation of its domain of dependence $D(\Gamma)$ by constant mean curvature surfaces with bounded second fundamental form. Moreover, these surfaces provide a family of quasi-conformal extensions of $\phi$. This answers Question 8.3 in \cite{questions}. 
\end{abstract} 

\maketitle

\setcounter{tocdepth}{1}
\tableofcontents

\section*{Introduction}\label{Introduction}
The study of Anti-de Sitter geometry has grown in interest since the pioneering work of Mess, who pointed out many analogies betweeen hyperbolic geometry and many connections to Teichm\"uller theory \cite{Mess}. \\
\indent More recently, with the work of Bonsante and Schlenker \cite{maxsurface} and of Bonsante and Seppi \cite{BonSep}, Anti-de Sitter geometry turned out to be a useful tool to construct quasi-conformal extensions of quasi-symmetric homeomorphisms of the unit disc. Indeed, the graph of a quasi-symmetric map $\phi: S^{1} \rightarrow S^{1}$ describes a curve (called quasi-circle) $c_{\phi}$ on the boundary at infinity of the $3$-dimensional Anti-de Sitter space $AdS_{3}$. Bonsante and Schlenker proved that a smooth surface $S$ (satisfying some technical conditions) with asymptotic boundary $c_{\phi}$ defines a quasi-conformal extension of $\phi$. Moreover, some remarkable properties of the quasi-conformal extension can be deduced from the geometry of the surface itself: for example, when the surface $S$ is maximal (i.e. it has vanishing mean curvature) the quasi-conformal extension induced by $S$ is minimal Lagrangian (\cite{maxsurface}); when $S$ is a smooth convex $\K$-surface, the corresponding quasi-conformal extension is a landslide (\cite{BonSep}); when the surface $S$ is the past-convex boundary of the convex-hull of $c_{\phi}$, the quasi-symmetric homeomorphism $\phi$ is exactly the boundary map of the earthquake $E^{2\lambda}: \h^{2} \rightarrow \h^{2}$, where $\lambda$ is the pleating locus of $S$ (\cite{Mess}).\\
\indent In this paper we study quasi-conformal extensions induced by constant mean curvature surfaces (in brief $H$-surfaces). The first problem we address is the existence of an $H$-surface with a given quasi-circle $\Gamma$ as asymptotic boundary, thus generalizing the work of Bonsante and Schlenker (\cite{maxsurface}) for maximal surfaces. We will prove that for each $H \in \R$, there exists an $H$-surface with asymptotic boundary $\Gamma$ and bounded principal curvatures. Although the technical part of the proof is based on the same apriori estimates as in \cite{maxsurface}, the starting point for the construction of this $H$-surface is different, thus obtaining a somehow new proof also in the case when $H=0$. Namely, we construct this $H$-surface $S_{H}$ as a limit of $H$-surfaces $(S_{H})_{n}$, with asymptotic boundary $\Gamma_{n}$, with the property that $\Gamma_{n}$ is the graph of a quasi-symmetric homeomorphism conjugating two cocompact Fuchsian groups and $\Gamma_{n}$ converges to $\Gamma$ in the Hausdorff topology. The existence of this approximating sequence $(S_{H})_{n}$ is a consequence of some results in \cite{CMCfoliation} and \cite{BonSep}.  \\

Moreover, extending the results of \cite{CMCfoliation}, we prove the following: \\

{\bf Theorem 3.1}{\it \ Given a quasi-circle $\Gamma \subset \partial_{\infty}AdS_{3}$, there exists a foliation by constant mean curvature surfaces $S_{H}$ for $H \in (-\infty, +\infty)$ of the domain of dependence $D(\Gamma)$.}\\

In the second part of the paper, we estimate the principal curvatures of a constant mean curvature surface. Those results will then be used to prove the uniqueness of the foliation (Theorem \ref{thm:uniqueness}) and to prove that each $H$-surface bounding a quasi-circle induces a quasi-conformal extension of a quasi-symmetric homeomorphism (Proposition \ref{prop:Hextension}).  

\subsection*{Outline of the paper} In Section 1 we describe some geometric models of the $3$-dimensional Anti-de Sitter space. In Section 2 we recall how to associate to a smooth surface bounding the graph of a quasi-symmetric homeomorphism $\phi$ a quasi-conformal extension of $\phi$ to the unit disc, following \cite{maxsurface}. Section 3 is devoted to the proof of Theorem \ref{thm:existencefoliation}. The uniqueness of the foliation is proved in Section 5, using some estimates on the principal curvatures of an $H$-surface proved in Section 4. In Section 6 we prove Proposition \ref{prop:Hextension}.

\section{Anti-de Sitter space}\label{sec:model}
\indent The $3$-dimensional Anti-de Sitter space $AdS_{3}$ is the Lorentzian analogue of the hyperbolic space, i.e. it is the local model for $3$-dimensional Lorentzian manifolds of constant sectional curvature $-1$. In this section we introduce two geometric models that we will use in the sequel: as the interior of a quadric in $\R^{3}$ and as a product $\h^{2} \times S^{1}$. The main references for the material covered here are \cite{Mess} and \cite{maxsurface}.

\subsection*{The Klein model of Anti-de Sitter space} Consider $\R^{4}$ endowed with the symmetric bilinear form of signature $(2,2)$
\[
	\langle x,y\rangle_{2,2}=-x_{0}y_{0}+x_{1}y_{1}+x_{2}y_{2}-x_{3}y_{3} \ .
\]
The $3$-dimensional Anti-de Sitter space is the quadric 
\[
	\widehat{AdS}_{3}=\{ x \in \R^{4} \ | \ \langle x,x\rangle_{2,2} =-1 \}
\]
endowed with the Lorentzian metric induced by the restriction of the bilinear form $\langle \cdot, \cdot\rangle_{2,2}$ to its tangent spaces. Given a point $p \in \widehat{AdS}_{3}$ and a tangent vector $v \in T_{p}\widehat{AdS}_{3}$, we say that
\begin{itemize}
\item $v$ is space-like if $\langle v,v\rangle_{2,2}>0$ \ ;
\item $v$ is light-like if $\langle v,v\rangle_{2,2}=0$ \ ;
\item $v$ is time-like if $\langle v,v\rangle_{2,2}<0$ \ .
\end{itemize}
In a similar way we distinguish space-like, time-like and light-like geodesics. These are obtained by intersecting a plane through the origin with the quadric.  Given two points $p, q\in \widehat{AdS}_{3}$, we say that they are causally related if there exists a non-space-like curve connecting them. Under this condition, we define the distance between $p$ and $q$ as 
\[
	d_{AdS}(p,q)=\sup_{\gamma}\int_{0}^{1}\sqrt{-g_{AdS}(\dot{\gamma}, \dot{\gamma})} ds \ ,
\]
where the supremum is taken over all possible non-space-like curves $\gamma:[0,1] \rightarrow \widehat{AdS}_{3}$ such that $\gamma(0)=p$ and $\gamma(1)=q$. A spacelike plane is a totally geodesic plane whose induced metric is Riemannian. Moreover, every space-like plane is isometric to the hyperbolic plane. More in general, a space-like surface is a $2$-dimensional sub-manifold whose tangent plane at every point is space-like. \\
\indent A time-orientation is a choice of a never-vanishing time-like vector field on $\widehat{AdS}_{3}$. The isometry group of orientation and time-orientation preserving isometries of $\widehat{AdS}_{3}$  is identified with $SO_{0}(2,2)$, as they are linear transformation of $\R^{4}$ preserving the bilinear form $\langle \cdot, \cdot \rangle_{2,2}$. \\
\\
\indent The Anti-de Sitter space can be more easily visualised by considering its image in the projective space. If we denote with $\pi: \R^{4}\setminus \{0\} \rightarrow \R\Pp^{3}$ the canonical projection we define
\[
	AdS_{3}=\pi(\{x \in \mathbb{R}^{4} \ | \ \langle x,x\rangle_{2,2}<0\}) \ .
\]
It can be easily verified that $\pi: \widehat{AdS}_{3} \rightarrow AdS_{3}$ is a double cover, hence we can endow $AdS_{3}$ with the unique Lorentzian structure that makes $\pi$ a local isometry. \\
\indent In the affine chart $U_{3}=\{x \in \R^{4} \ | \ x_{3}\ne 0\}$, the Anti-de Sitter space fills the interior of the double-ruled quadric $Q$ of equation $-x^{2}+y^{2}+z^{2}=1$. In analogy with the Klein model of the hyperbolic space, geodesics and totally geodesics planes are obtained by intersecting the interior of the quadric with affine lines and planes.\\
\indent Moreover, the projective duality in $\R\Pp^{3}$ induces a duality between points and totally geodesic space-like planes in $AdS_{3}$.  \\
\\
\indent It is then natural to define the boundary at infinity of $AdS_{3}$ as 
\[
	\partial_{\infty}AdS_{3}=\pi(\{ x \in \R^{4} \ |  \ \langle x,x\rangle_{2,2}=0 \}) \ .
\]
In the affine chart $U_{3}$, this corresponds to the double-ruled quadric $Q$ of equation $-x^{2}+y^{2}+z^{2}=1$. Using the two rulings, which we call left and right, we can identify the boundary at infinity of $AdS_{3}$ with two copies of $S^{1}$ in the following way.  Fix a spacelike plane $P$ in $AdS_{3}$. Its boundary at infinity is a circle. Given a point $q\in \partial_{\infty}AdS_{3}$, we can associate biunivocally a couple $(\pi_{l}(q), \pi_{r}(q))\in \partial_{\infty}P \times \partial_{\infty}P$ by following the left and right ruling of the quadric $Q$.\\ 
\\
\indent Given a continuous curve $\Gamma$ in $\partial_{\infty}AdS_{3}$, we say that $\Gamma$ is weakly acausal if for every point $p\in \Gamma$, there exists a neighborhood $U$ of $p$ in $\partial_{\infty}AdS_{3}$ such that $U\cap \Gamma$ is contained in the complement of the regions of $U$ which are connected to $p$ by time-like curves. 
\begin{defi} The domain of dependence $D(\Gamma)$ of a weakly acausal curve $\Gamma$ is the union of points $p \in AdS_{3}$ such that the dual plane $p^{*}$ is disjoint from $\Gamma$. 
\end{defi}

It turns out that the domains of dependence are always contained in an affine chart and only admit light-like support planes.\\
 
\indent We will also use the following notion:
 
\begin{defi} The convex hull $C(\Gamma)$ of a weakly acausal curve $\Gamma$ is the smallest closed convex subset which contains $\Gamma$.
\end{defi}

\subsection*{A product model for Anti-de Sitter space}
The map 
\begin{align*}
	F: \h^{2} \times S^{1} &\rightarrow \widehat{AdS}_ {3}\\
			(x_{0}, x_{1}, x_{2}, e^{i\theta}) &\mapsto (x_{0}\cos(\theta), x_{1}, x_{2}, x_{0}\sin(\theta))
\end{align*}
is a diffeomorphism (we are considering the hyperboloid model of $\h^{2}$). Hence $\h^{2} \times S^{1}$ is isometric to the Anti-de Sitter space, when endowed with the pull-back metric
\[
	F^{*}g_{AdS_{3}}(x, e^{i\theta})=g_{\h^{2}}-x_{0}^{2}d\theta^{2} \ .
\] 
It is sometimes convenient to consider the universal covering $\widetilde{AdS}_{3}\cong \h^{2}\times \R$ endowed with the metric
\[
	 g_{\widetilde{AdS}}=g_{\h}^{2}-x_{0}^{2}dt^{2} \ .
\]
We notice that $\frac{\partial}{\partial t}$ is a time-like Killing vector field. We will denote 
\[
	\chi^{2}=-\left\| \frac{\partial}{\partial t}\right\|^{2}
\]
and 
\[
	\grad t=-\frac{1}{\chi^{2}}\frac{\partial}{\partial t} \ .
\]
\indent This model is particularly useful to study space-like surfaces. In fact, space-like surfaces are graphs of functions (Prop. 3.2 in \cite{maxsurface})
\begin{align*}
	u: \h^{2} &\rightarrow \R \\
		x &\mapsto u(x)\ .
\end{align*}
Moreover, the space-like condition provides a uniform bound on the gradient of $u$. For instance, let us consider the function $\hat{u}$ on $\h^{2}\times \R$ given by
\[
	\hat{u}(x,t)=u(x) \ .
\]
The space-like surface described by the function $\hat{u}$ is defined by the equation $\hat{u}(x)-t=0$. This is space-like if and only if the normal vector at each point
\[
	\nu= -\chi^{2}\grad t-\grad(\hat{u})
\]
is time-like. We deduce the uniform bound
\[
	\|\grad(u)\|^{2} < \frac{1}{\chi^{2}}
\]
on the gradient of the function $u$. 

\section{Quasi-symmetric homeomorphisms and quasi-conformal extensions}\label{sec:extensions}
In this section we recall some well-known results about quasi-symmetric homeomorphisms of $S^{1}$. The graph of a quasi-symmetric homeomorphism $\phi$ describes a curve $c_{\phi}$ on the boundary at infinity of $AdS_{3}$ and, under some additional conditions, smooth negatively curved surfaces bounding $c_{\phi}$ provide quasi-conformal extensions of $\phi$ (\cite{maxsurface}, \cite{SchKra}). We recall here briefly this construction.\\
\\
\indent In Section \ref{sec:model} we have seen that it is possible to identify the boundary at infinity of the $3$-dimensional Anti-de Sitter space with $S^{1} \times S^{1}$. With this identification, we can represent the graph of a homeomorphism $\phi: S^{1} \rightarrow S^{1}$ as a curve on the boundary at infinity of $AdS_{3}$, namely
\[
	c_{\phi}=\{ (x,\phi(x))\in \partial_{\infty}AdS_{3} \ | \ x\in S^{1}\} \ .
\]
\\
\indent A homeomorphism $\phi: S^{1} \rightarrow S^{1}$ is quasi-symmetric if there exists a constant $C>0$ such that
\[
	\sup_{Q}| \log|cr(\phi(Q))|| \leq C \ ,
\]
where the supremum is taken over all quadruple $Q$ of points in $S^{1}$ with cross ratio $cr(Q)=-1$, and we use the following definition of cross-ratio
\[
	cr(x_{1}, x_{2}, x_{3}, x_{4})=\frac{(x_{4}-x_{1})(x_{3}-x_{2})}{(x_{2}-x_{1})(x_{3}-x_{4})} \ .
\] 

\begin{defi} An acausal curve $\Gamma\subset \partial_{\infty}AdS_{3}$ is a quasi-circle, if it is the graph of a quasi-symmetric homeomorphism. 
\end{defi}

\begin{oss} It follows from the identification between the boundary at infinity of $AdS_{3}$ and $S^{1} \times S^{1}$ that an acausal curve $\Gamma$ is a quasi-circle if and only if $\phi=\pi_{r}\circ \pi_{l}^{-1}$ is quasi-symmetric. Moreover, $\Gamma$ is the graph of $\phi$.
\end{oss}

\indent An orientation-preserving homeomorphism $f:D^{2} \rightarrow D^{2}$ is quasi-conformal if $f$ is absolutely continuous on lines and there exists a constant $k<1$ such that
\[
	|\mu_{f}|=\left| \frac{\overline{\partial}f}{\partial f}\right| \leq k  \ .
\]
A map with this property can also be called $K$-quasi-conformal, where
\[
	K=\frac{1+\|\mu_{f}\|_{\infty}}{1-\|\mu_{f}\|_{\infty}} \in [1, +\infty)
\]

The relation between quasi-symmetric homeomorphisms of the circle and quasi-conformal maps of the unit disc is provided by the following well-known theorem:

\begin{teo}[\cite{Ahlfors}] Every quasi-conformal map $\Phi: D^{2}\rightarrow D^{2}$ extends to a quasi-symmetric homeomorphism of $S^{1}$. Conversely, any quasi-symmetric homeomorphism $\phi: S^{1}\rightarrow S^{1}$ admits a quasi-conformal extension to $D^{2}$. 
\end{teo}

If we represent the graph of a quasi-symmetric homeomorphism $\phi$ as a curve $c_{\phi}$ on the boundary at infinity of $AdS_{3}$, in \cite{SchKra} it is explained how to obtain quasiconformal extensions of $\phi$ using smooth, negatively curved, space-like surfaces with boundary at infinity $c_{\phi}$. The construction goes as follows. We fix a totally geodesic space-like plane $P_{0}$. Let $S$ be a space-like, negatively-curved surface embedded in $AdS_{3}$. Let $\tilde{S}'\subset U^{1}AdS_{3}$ be its lift into the unit tangent bundle of $AdS_{3}$ and let $p:\tilde{S}' \rightarrow \tilde{S}$ be the canonical projection. For any point $(x,v)\in \tilde{S}'$, there exists a unique space-like plane $P$ in $AdS_{3}$ orthogonal to $v$ and containing $x$. We define two natural maps $\Pi_{\infty, l}$ and $\Pi_{\infty,r}$ from $\partial_{\infty}P$ to $\partial_{\infty}P_{0}$, sending a point $x\in \partial_{\infty}P$ to the intersection between $\partial_{\infty}P_{0}$ and the unique line of the left or right foliation of $\partial_{\infty}AdS_{3}$ containing $x$. Since these maps are projective, they extend to hyperbolic isometries $\Pi_{l}, \Pi_{r}:P \rightarrow P_{0}$. We then define the map $\Phi=\Pi_{r}\circ \Pi_{l}^{-1}$. This map is always a local diffeomorphism of $\h^{2}$ when the surface is negatively curved, as the differentials of the maps $\Pi_{l}$ and $\Pi_{r}$ are given by
\[
	d\Pi_{l}=E+JB \ \ \ \ \ \ \ \ d\Pi_{r}=E-JB \ .
\]
On the other hand, $\Phi$ is not always a global diffeomorphism, but the following lemma gives some sufficient conditions on the surface $S$ which guarantee that $\Phi$ is proper (and hence a homeomorphism) and that its boundary value coincides with $\phi$:

\begin{lemma}[Lemma 3.18 in \cite{SchKra}]\label{lm:criterioestensione} Let $S$ be a space-like, negatively curved surface in $AdS_{3}$ whose boundary at infinity $\Gamma$ does not contain any light-like segment. Suppose that there is no sequence of points $x_{n}$ on $S$ such that the totally geodesic planes $P_{n}$ tangent to $S$ at $x_{n}$ converge to a light-like plane $P$ whose past end-point and future end-point are not in $\Gamma$. Then for any sequence of points $x_{n} \in S$ converging to $x\in \Gamma$ we have that $\Pi_{l}(x_{n})\to \pi_{l}(x)$ and $\Pi_{r}(x_{n})\to \pi_{r}(x)$.
\end{lemma}

\begin{oss}\label{oss:criterioestensione} As noticed in \cite{SchKra}, the hypothesis of Lemma \ref{lm:criterioestensione} are satisfied in case of a smooth, convex, space-like surface bounding a quasi-circle.
\end{oss}

\section{Existence of a CMC foliation}\label{sec:existence}
This section is devoted to the proof of the following:

\begin{teo}\label{thm:existencefoliation} Given a quasi-circle $\Gamma \subset \partial_{\infty}AdS_{3}$, there exists a foliation by constant mean curvature surfaces $S_{H}$ for $H \in (-\infty, +\infty)$ of the domain of dependence $D(\Gamma)$.
\end{teo}

As outlined in the introduction, the main idea to construct a constant mean curvature surface with a given quasi-circle as boundary at infinity is a process by approximation. In fact, as a consequence of the work \cite{CMCfoliation}, the existence (and uniqueness) of a constant mean curvature foliation is known for a particular class of quasi-circles:

\begin{teo}[Thm 1.1 in \cite{CMCfoliation}]\label{thm:foliationBBZ} Let $\Gamma$ be a quasi-circle which is the graph of a quasi-symmetric homeomorphism that conjugates two cocompact Fuchsian groups. Then there exists a unique foliation by equivariant H-surfaces of the domain of dependence of $\Gamma$, where $H \in (-\infty, +\infty)$.
\end{teo}

Moreover, by a recent result in \cite{BonSep}(Lemma 7.2), every quasi-circle can be uniformly approximated by a sequence of quasi-circles, which are the graphs of quasi-symmetric homeomorphisms conjugating two cocompact Fuchsian groups. \\

\indent Therefore, given a quasi-circle $\Gamma$, we will consider a sequence of quasi-circles $\Gamma_{n}$, which are the graphs of quasi-symmetric homeomorphisms conjugating two cocompact Fuchsian groups, converging in the Hausdorff topology to $\Gamma$. For each $H \in (-\infty, +\infty)$, Theorem \ref{thm:foliationBBZ} provides a sequence of $H$-surfaces $(S_{H})_{n}$ with boundary at infinity $\Gamma_{n}$. In this section we will prove that the sequence $(S_{H})_{n}$ converges $C^{\infty}$ on compact sets to an $H$-surface $(S_{H})_{\infty}$ with boundary at infinity $\Gamma$. This will give us the existence of a surface with given boundary at infinity and given constant mean curvature $H$ for every $H \in (-\infty, +\infty)$. We will then prove that these surfaces provide a foliation of the domain of dependence of $\Gamma$. \\

\indent We first recall some definitions. In the universal cover of $AdS_{3}$, given a space-like surface $M$, we recall that $M$ is the graph of a function $u: \h^{2} \rightarrow \R$. We define the gradient function with respect to the vector field $T=-\chi \nabla t$ as
\[
	v_{M}=-\langle \nu, T \rangle=\frac{1}{\sqrt{1-\chi^{2}|\nabla u|^{2}}} 
\]
where $\nu$ is the unit future-oriented normal vector field. The shape operator of $M$ is defined by 
\[
	B(X)=-\nabla_{X}\nu
\]
for every vector field $X$ on $M$. The mean curvature of $M$ is 
\[
	H=\frac{\trace(B)}{2} \ .
\]
We can write explicitely a formula for the mean curvature of $M$, in terms of $u$ and $T$ (see e.g. \cite{Bartnik}):
\begin{equation}\label{eq:meancurvature}
	H=\frac{1}{2v_{M}}(\dive_{M}(\chi\grad_{M} u)+\dive_{M}T) \ . 
\end{equation}

We will need the following a-priori estimate for the gradient function $v_{M}$, which is a consequence of the work of Bartnik \cite{Bartnik}. Given a point $p \in \widetilde{AdS}_{3}$, we denote with $I^{+}(p)$ the set of points in the future of $p$, and similarly with $I^{-}(p)$ the set of points in the past of $p$. We will indicate with $I^{+}_{\epsilon}(p)$ the set of points in the future of $p$ at distance at least $\epsilon$. We have the following:

\begin{lemma}\label{lm:stime} Let $p \in \widetilde{AdS}_{3}$ and $\epsilon>0$. Let $K$ be a compact domain contained on a region where the covering map $\pi: \widetilde{AdS}_{3} \rightarrow AdS_{3}$ is injective. There exists a constant $C=C(p, \epsilon, K)$ such that for every $H$-surface $M$ that verifies
\begin{itemize}
	\item $\partial M \cap I^{+}(p)=\emptyset$;
	\item $M\cap I^{+}(p)\subset K$,
\end{itemize}
we have that 
\[
	\sup_{M\cap I^{+}_{\epsilon}(p)} v_{M} <C \ .
\]
\end{lemma}
\begin{proof} Consider the time function
\[
	\tau(x)=d_{AdS}(x,p)-\frac{\epsilon}{2} \ ,
\]
where $d_{AdS}(x,p)$ is the Lorentzian distance between $x$ and $p$. This function is smooth on $V=K \cap I^{+}(p)$. By assumption on $M$, the region $M\cap V$ contains the set $\{ \tau \geq 0\}\cap M$ and $M\cap I^{+}_{\epsilon}(p)$ is contained in $V$. We can thus apply Theorem 3.1 in \cite{Bartnik} and conclude that
\[
	\sup_{M \cap I_{\epsilon}^{+}(p)} v_{M} < C \ ,
\]
where the constant $C$ depends on the $C^{2}$ norms of $t$ and $T$ and on the $C^{0}$ norm of the Ricci tensor on the domain $V \cap \{ \tau \geq 0\}$ with respect to a reference Riemannian metric. 
\end{proof}

We will also need the following result that provides some barriers for constant mean curvature surfaces in Anti-de Sitter manifolds:

\begin{prop}\label{prop:barriers} Let $\Sigma$ be a space-like surface with constant mean curvature $H$ embedded in $AdS_{3}$ with boundary at infinity a quasi-circle $\Lambda$. Suppose that $\Sigma$ the lift of a compact surface embedded in a GHMC Anti-de Sitter manifold. Then there exists $\kappa \leq -1$ such that $\Sigma$ is in the past of the past-convex surface $S_{\K}^{+}$ and in the future of the future-convex surface $S_{\K}^{-}$ with constant Gauss curvature $\K$ and asymptotic boundary $\Lambda$.
\end{prop}
\begin{proof} If $H=0$, the statement holds, since a maximal surface is contained in the convex hull of $\Lambda$. For the other values of $H$ we choose $\kappa <-1$ such that $\sqrt{-1-\kappa}>|H|$. We claim that the past-convex space-like surface $S_{\K}^{+}$ with constant curvature $\K$, whose existence is proved in \cite{BonSep}, must be in the future of $\Sigma$. If not, the surfaces $\Sigma$ and $S_{\K}^{+}$ would intersect transversely, but, since constant curvature surfaces provide a foliation of $D(\Lambda)\setminus C(\Lambda)$, there would exist a $\K'<\K$ such that the surface $S_{\K'}$ with constant Gauss curvature $\K'$ is tangent to $\Sigma$ at a point $x$. By the Maximum Principle, the mean curvature of $S_{\K'}$ at $x$ must be smaller than the mean curvature of $\Sigma$ at $x$, but this is impossible for our choice of $\K'$. \\
\indent With a similar reasoning we obtain that the future-convex space-like surface $S_{\K}$ must be in the past of $\Sigma$. 
\end{proof} 

We have now all the ingredients to prove the existence of an $H$-surface with given asymptotic boundary. Let $\Gamma$ be a quasi-circle on $AdS_{3}$ and let $\Gamma_{n}$ be a sequence of quasi-circles converging to $\Gamma$ in the Hausdorff topology that are the graphs of quasi-symmetric homeomorphisms that conjugate two cocompact Fuchsian groups. Let $(S_{H})_{n}$ be the $H$-surface with asymptotic boundary $\Gamma_{n}$ provided by Theorem \ref{thm:foliationBBZ}. 

\begin{teo}\label{thm:existenceHsurface} The sequence of $H$-surfaces $(S_{H})_{n}$ converges $C^{\infty}$ on compact sets to an $H$-surface $(S_{H})_{\infty}$ with boundary at infinity $\Gamma$.
\end{teo}
\begin{proof} We consider their lifts $(\tilde{S}_{H})_{n}$ to the universal cover $\widetilde{AdS}_{3}$. We denote with $(\tilde{\Sigma}_{\K}^{\pm})_{n}$ the lifts of the two constant curvature surfaces provided by Proposition
\ref{prop:barriers}. In general, we will use the notation with a tilda to indicate the lift of an object to the universal cover. By Theorem 7.8 in \cite{BonSep}, the sequence $(\tilde{\Sigma}_{\K}^{\pm})_{n}$ converges to constant curvature surfaces $\tilde{\Sigma}_{\K}^{\pm}$ with boundary at infinity $\tilde{\Gamma}$. We denote with $K'$ the domain 
\[
	K'=I^{-}(\tilde{\Sigma}_{\K}^{+})\cap I^{+}(\tilde{\Sigma}_{\K}^{-}) \ .
\]
For any point $\tilde{p} \in \widetilde{D(\Gamma)}\cap I^{-}(\tilde{\Sigma}_{\K}^{-})$, we choose $\epsilon(\tilde{p})$ such that the family $\{ I^{+}_{\epsilon(\tilde{p})}(\tilde{p}) \cap K'\}$ is an open covering of $K'$. Since 
\[
	K'_{n}=I^{-}((\tilde{\Sigma}_{\K}^{+})_{n})\cap I^{+}((\tilde{\Sigma}_{\K}^{-})_{n})
\]
converges to $K'$, there exists an $n_{0}$ such that for every $n \geq n_{0}$ the closed set
\[
	K=\overline{\bigcup_{n \geq n_{0}}K'_{n}}
\]
is contained in the open covering $\cup\{ I^{+}_{\epsilon(\tilde{p})}(\tilde{p})\}$ constructed above.     \\
Given a number $R>0$ we denote with $B_{R}$ the ball of radius $R$ in $\h^{2}$ centered at the origin in the Poincaré model. The intersection $(B_{R} \times \R)\cap K$ is compact, so there is a finite number of points $\tilde{p}_{1}, \dots, \tilde{p}_{m}$ such that 
\[
	(B_{R}\times \R) \cap K \subset \bigcup_{j=1}^{m} I^{+}_{\epsilon(\tilde{p}_{j})}(\tilde{p}_{j})  \ .
\]
We notice that, since $\tilde{p}_{j} \in \widetilde{D(\Gamma)}$, the intersection $I^{+}(\tilde{p}_{j}) \cap \widetilde{D(\Gamma)}$ is compact. Moreover, since the plane dual to $\tilde{p}_{j}$ is disjoint from $\tilde{\Gamma}$ for every $j=1, \dots, m$, if we choose $n_{0}$ big enough, the same is true for $\Gamma_{n}$ for every $n \geq n_{0}$, because $\Gamma_{n}$ converges to $\Gamma$ in the Hausdorff topology. In this way we can ensure that the set $K_{j}=\overline{I^{+}(\tilde{p}_{j})}\cap K$ is compact and contained on a region where the covering map $\pi: \widetilde{AdS}_{3} \rightarrow AdS_{3}$ is injective. By Lemma \ref{lm:stime}, there is a constant $C_{j}$ such that
\[
	\sup_{M \cap I^{+}_{\epsilon(\tilde{p}_{j})}(\tilde{p}_{j})} v_{M} < C_{j}
\]
for every constant mean curvature surface $M$ satisfying 
\begin{enumerate}[(i)]
\item $\partial M \cap I^{+}(\tilde{p}_{j}) = \emptyset$ \ ;
\item $M \cap I^{+}(\tilde{p}_{j})$ is contained in $K_{j}$ \ .
\end{enumerate}
Condition (i) is clearly satisfied for $n \geq n_{0}$ by definition of the set $K$. As for Condition (ii), the boundary of $(\tilde{S}_{H})_{n}$ is disjoint from the future of $\tilde{p}_{j}$ for every $j=1, \dots m$ due to our choice of $n_{0}$. \\
If we denote with $v_{n}$ the gradient function associated to the surface $(\tilde{S}_{H})_{n}$, it follows that
\begin{equation}\label{stima}
	\sup_{(\tilde{S}_{H})_{n} \cap (B_{R} \times \R)} v_{n} \leq \max\{ C_{1}, \dots , C_{m}\}
\end{equation}
for every $n \geq n_{1}$.\\
\indent We deduce that for every $R$ there is a constant $C(R)$ such that the gradient function $v_{n}$ is bounded by $C(R)$ for $n$ sufficiently large.\\
\indent Let $u_{n}: \h^{2} \rightarrow \R$ such that $(\tilde{S}_{H})_{n}$ are the graph of the function $u_{n}$. By comparing Equation (\ref{eq:meancurvature}) with estimate (\ref{stima}), we see that the restriction of $u_{n}$ on $B_{R}$ is the solution of a uniformly elliptic quasi-linear PDE with bounded coefficients. Since $u_{n}$ and the graident $\nabla u_{n}$ are uniformly bounded on $B_{R}$ (see Section \ref{sec:model}), by elliptic regularity the norms of $u_{n}$ in $C^{2, \alpha}(B_{R-1})$ are uniformly bounded. We can thus extract a subsequence $u_{n_{k}}$ which converges $C^{2}$ to some function $u_{\infty}$ on compact sets. Since $u_{\infty}$ is a $C^{2}$-limit of solutions of Equation (\ref{eq:meancurvature}), it is still a solution and its graph $S_{H}$ has constant mean curvature $H$. \\    
\indent The boundary at infinity of $S_{H}$ coincides with $\Gamma$ because it is the Hausdorff limit of the curves $\Gamma_{n}$, which converge to $\Gamma$, by construction.
\end{proof}

Moreover, we can deduce that the principal curvatures of the surface $(S_{H})_{\infty}$ are uniformly bounded, due to the following:  
\begin{lemma} Let $S$ be an $H$-surface embedded in $AdS_{3}$, which is the lift of a space-like compact surface embedded into a GHMC $AdS_{3}$ manifold. Then the principal curvatures $\mu$ and $\mu'$ of $S$ are bounded by some constant depending only on $H$. More precisely,
\[
	H\leq \lambda_{1} \leq \sqrt{H^{2}+1}+H \ \ \ \ \text{and} \ \ \ \ \ -\sqrt{H^{2}+1}+H \leq \lambda_{2} \leq H \ .
\]
\end{lemma}
\begin{proof} Let $B$ be the shape operator of $S$. We consider $B_{0}=B-HE$, the traceless part of $B$ (here $E$ is the identity operator). Since $H$ is constant, the operator $B_{0}$ is Codazzi. Let $e_{1}$ and $e_{2}$ be tangent vectors in a orthonormal frame of $S$ that diagonalise $B_{0}$. Since $B_{0}$ is traceless, the eigenvalues are opposite, and we will denote with $\lambda\geq 0$ the eigenvalue of $e_{1}$. Let $\omega$ be $1$-form connection of the Levi-Civita connection $\nabla$ for the induced metric on $S$, defined by the relation
\[
	\nabla_{x}e_{1}=\omega(x)e_{2} \ .
\]
The Codazzi equation for $B_{0}$ can be read as follows,
\[
	\begin{cases} \lambda \omega(e_{1})=d\lambda(e_{2}) \\
				  \lambda \omega(e_{2})=-d\lambda(e_{1})  \ .
	\end{cases}
\]
If we define $\beta=\log(\lambda)$ we obtain
\[
	\begin{cases}  \omega(e_{1})=d\beta(e_{2}) \\
				   \omega(e_{2})=-d\beta(e_{1})
	\end{cases}
\]
Moreover, if we denote with $\K$ the Gaussian curvature of $S$, we have
\[
	-K=d\omega(e_{1}, e_{2})=e_{1}(\omega(e_{2}))-e_{2}(\omega(e_{1}))-\omega([e_{1}, e_{2}])=\Delta \beta \ , 
\]
where $\Delta$ is the Laplacian that is positive at the points of local maximum. On the other hand by the Gauss equation,
\[
	-K=\det(B)+1=\det(B_{0}+HE)+1=\det(B_{0})+H^{2}+1=-e^{2\beta}+H^{2}+1 \ .
\]
Since the surface $S$ is the lift of a compact surface, the function $\beta$ has maximum at a point $x_{0}$. By the fact that $\Delta \beta(x_{0})\geq0$, we deduce that
\[
	\lambda=e^{\beta}\leq \sqrt{H^{2}+1} \ .
\]
Since the eigenvalues of $B$ are $\mu=\lambda+H$ and $\mu'=-\lambda+H$, we obtain the claim.
\end{proof}

We have thus found for every value of $H \in \R$ a constant mean curvature surface $S_{H}$, with bounded principal curvatures, bounding a given quasi-circle $\Gamma$ at infinity. We conclude this section by showing that these surfaces provide a foliation of the domain of dependence $D(\Gamma)$.
\begin{prop} Let $S_{H}$ be the family of $H$-surfaces provided by Theorem \ref{thm:existenceHsurface} with boundary at infinity $\Gamma$. Then $\{S_{H}\}_{H \in \R}$ foliates the domain of dependence $D(\Gamma)$.
\end{prop}
\begin{proof} We first show that if $H_{1} < H_{2} \in \mathbb{Q}$, then $S_{H_{1}}$ and $S_{H_{2}}$ are disjoint and $S_{H_{1}}$ is in the past of $S_{H_{2}}$. By construction $S_{H_{1}}$ and $S_{H_{2}}$ are $C^{\infty}$ limits of the sequences $(S_{H_{1}})_{n}$ and $(S_{H_{2}})_{n}$ of constant mean curvature surfaces with boundary at infinity $\Gamma_{n}$ which is a graph of a quasi-symmetric homeomorphism that conjugates two cocompact Fuchsian groups. By Theorem \ref{thm:foliationBBZ}, they are leaves of the constant mean curvature foliation of $D(\Gamma_{n})$ and, in particular, $(S_{H_{1}})_{n}$ is in the past of $(S_{H_{2}})_{n}$ for every $n$. Hence, the same holds for $S_{H_{1}}$ and $S_{H_{2}}$. This shows that they cannot intersect transversely. Moreover, it is not possible that $S_{H_{1}}$ and $S_{H_{2}}$ are tangent at one point, since the trace of the shape operator of $S_{H_{2}}$ is bigger than the trace of the shape operator of $S_{H_{1}}$ and this would contradict the Maximum Principle (see Lemma \ref{lm:maxprinciple}). \\
\indent We now show that if we take two sequence of rational numbers $H_{k}'$ converging increasingly to $H\in \R$ and $H_{k}''$ converging decreasingly to $H\in \R$, then $S_{H_{k}'}$ and $S_{H_{k}''}$ converge to the same limit $S_{H}$. We first notice that the limits 
\[
	S_{H}'=\lim_{H_{k}' \to H}S_{H_{k}'} \ \ \ \ \ \ \text{and} \ \ \ \ \ S_{H}''=\lim_{H_{k}'' \to H}S_{H_{k}''}
\]
exist by a similar argument as in the proof of Theorem \ref{thm:existenceHsurface}. Moreover, $S_{H}'$ must be in the past of $S_{H}''$. Suppose by contradiction that $S_{H}'$ and $S_{H}''$ are distinct. Let $U$ be an open set contained in the past of $S_{H}''$ and in the future of $S_{H}'$. Since we know that the domain of dependence of $\Gamma_{n}$ is foliated by constant mean curvature surfaces, for $n$ large enough there exists a surface $S_{h_{n}}$ with constant mean curvature $h_{n} \in \mathbb{Q}$ and boundary at infinity $\Gamma_{n}$. By the uniform convergence on compact sets, for $n$ larger than some $n_{0}$, we must have $h_{n}<H_{k}''$ and $h_{n}>H_{k}'$ for every $k \in \N$, which gives the contradiction.\\
\indent So far we have proved that the $H$-surfaces $S_{H}$ provide a foliation of a subset of $D(\Gamma)$. We need to prove that  
\[
	\bigcup_{H \in \R} S_{H}= D(\Gamma) \ .
\]
Suppose by contradiction that there exists a point $p \in D(\Gamma)$ which does not lie in any of the surfaces $S_{H}$. Since the domain of dependence of $\Gamma_{n}$ converges to the domain of dependence of $\Gamma$, there exists a sequence of points $p_{n} \in D(\Gamma_{n})$ converging to $p$. Since $D(\Gamma_{n})$ is foliated by constant mean curvature surfaces, there exists a sequence $H_{n} \in \R$ such that $p_{n}\in S_{n}^{H_{n}} \subset D(\Gamma_{n})$. We claim that the sequence $H_{n}$ is bounded. We can assume that $H_{n}$ is positive for $n$ big enough. Since $p \in D(\Gamma)$, the boundary at infinity of the dual plane $p^{*}$ is disjoint from $\Gamma$. We choose a space-like plane $P$ in the future of $p^{*}$ with the following properties: the boundary at infinity of $P$ is disjoint from $\Gamma$ and $p \in D(P) \cap D(\Gamma)$. Since $\Gamma_{n}$ converges to $\Gamma$ in the Hausdorff topology, the asymptotic boundary of $P$ is disjoint from $\Gamma_{n}$ for $n$ big enough. Moreover, the surfaces $S_{n}^{H_{n}}$ converge to a nowhere time-like surface $S_{\infty}$ passing at $p$, because they are graphs of uniformly Lipschitz functions (see Section \ref{sec:model}). We notice that there exists a surface $\Sigma_{H_{0}}$ with constant mean curvature $H_{0}$ and boundary at infinity $\partial_{\infty}P$, such that $\Sigma_{H_{0}}$ intersects $S_{\infty}$ in a compact set and $p \in I^{-}(\Sigma_{H_{0}})$. This surface $\Sigma_{H_{0}}$ is obtained by taking equidistant surfaces from the space-like plane $P_{0}$. Let $F_{P}: D(P) \rightarrow \R$ the time function defined on the domain of dependence of $P$ such that each level set $F^{-1}(H)=\Sigma_{H}$ is a constant mean curvature surface with asymptotic boundary $\partial_{\infty}P$. It follows by construction that $F_{P}^{-1}((-\infty, H_{0})) \cap S_{\infty}$ is compact. Since the sequence $S_{n}^{H_{n}}$ converges uniformly on compact set to $S_{\infty}$, the same is true for $F_{P}^{-1}((-\infty, H_{0})) \cap S_{n}^{H_{n}}$, for $n$ sufficiently big. We define
\[
	H_{n}^{-}=\inf_{x \in S_{n}^{H_{n}}} F_{P}(x) \ .
\]
By the previous remarks $H_{n}^{-}\leq H_{0}$ for every $n$ sufficiently large, and it is assumed at some point $x_{n} \in S_{n}^{H_{n}}$. By construction, $S_{n}^{H_{n}}$ is tangent to $\Sigma_{H_{n}^{-}}$ at the point $x_{n}$ and $S_{n}^{H_{n}}$ is contained in the future of $\Sigma_{H_{n}^{-}}$. Therefore, by the Maximum Principle, we deduce that
\[
	H_{n} \leq H_{n}^{-} \leq H_{0}
\]
as claimed.
Therefore, there exist two real numbers $H^{+}$ and $H^{-}$ such that $S_{n}^{H_{n}}$ is in the past of $S_{n}^{H^{+}}$ and in the future of $S_{n}^{H^{-}}$ for every $n$. But the sequences $S_{n}^{H^{+}}$ and $S_{n}^{H^{-}}$ converge, by Theorem \ref{thm:existenceHsurface}, to constant mean curvature surfaces $S_{H^{+}}$ and $S_{H^{-}}$ with boundary at infinity $\Gamma$. But this implies that $p$ is contained in the subset of $D(\Gamma)$ foliated by the surfaces $S_{H}$ and this gives the contradiction. 
\end{proof}
This completes the proof of Theorem \ref{thm:existencefoliation}.

\section{Study of the principal curvatures of an $H$-surface}\label{sec:principalcurvatures}
The aim of this section is to give precise estimates for the principal curvatures of an $H$-surface with bounded second fundamental form. These results will then be used in Section \ref{sec:application} in order to associate to each surface in the foliation of the domain of dependence of a quasi-circle a quasi-conformal extension of the corresponding quasi-symmetric homeomorphism.\\
\\
\indent The main tool used to estimate the principal curvatures of an $H$-surface with bounded second fundamental form is the following compactness result for sequences of $H$-surfaces with bounded second fundamental form. This is a straightforward generalization of Lemma 5.1 in \cite{maxsurface}.

\begin{lemma}\label{lm:stability} Let $C>0$ be a fixed constant. Choose a point $x_{0} \in AdS_{3}$ and a future-oriented unit time-like vector $n_{0}\in T_{x_{0}}AdS_{3}$. There exists $r>0$ as follows. Let $P_{0}$ be the totally geodesic space-like plane orthogonal to $n_{0}$ at $x_{0}$. Let $D_{0}$ be the disk of radius $r$ centered at $x_{0}$ in $P_{0}$. Let $S_{n}$ be a sequence of $H$-surfaces 
containing $x_{0}$ and orthogonal to $n_{0}$ with second fundamental form bounded by $C$. After extracting a sub-sequence, the restrictions of $S_{n}$ to the cylinders above $D_{0}$ converge $C^{\infty}$ to an $H$-surface with boundary contained in the cylinder over $\partial D_{0}$. 
\end{lemma}
\begin{proof} For all $n$, the surface $S_{n}$ is the graph of a function $f_{n}$ over a totally geodesic plane $P_{n}$. The bound on the second form of $S_{n}$, along with the fact that $S_{n}$ is orthogonal to $n_{0}$ indicates that for some $r>0$, there exists $\epsilon>0$ such that
\[
	\| \nabla f_{n}\| \leq \frac{1-\epsilon}{\chi}
\]
on the disk of center $x_{0}$ and radius $r$. \\
Moreover, since the second fundamental form of $S_{n}$ are uniformly bounded, also the Hessian of $f_{n}$ are uniformly bounded by a constant depending on $r$.\\
\indent Therefore, we can extract a sub-sequence, still denoted with $f_{n}$, which converges $C^{1,1}$ to a function $f_{\infty}$ on the disk of center $x_{0}$ and radius $r$. We notice that, since the gradient of $f_{\infty}$ is uniformly bounded, the graph of $f_{\infty}$ is a space-like surface.\\
\indent By definition, the fact that $S_{n}$ are $H$-surfaces translates to the fact that $f_{n}$ is solution of Equation (\ref{eq:meancurvature}). Since $f_{\infty}$ is a $C^{1,1}$-limit of the sequence $f_{n}$, it is itself a weak solution of the same equation. By elliptic regularity, it follows that $f_{\infty}$ is $C^{\infty}$ and $f_{n}$ is actually converging to $f_{\infty}$ in the $C^{\infty}$ sense. Therefore, the graph of $f_{\infty}$ over the disk of radius $r$ is an $H$-surface, which is the $C^{\infty}$ limit of the restriction of the $H$-surfaces $S_{n}$ to the cylinder above the disk of radius $r$.
\end{proof}

We will need also a stability result for sequences of quasi-circles in $\partial_{\infty}AdS_{3}$. This will be a consequence of the following compactness property of quasi-symmetric homeomorphisms:
\begin{prop} Let $\phi_{n}:S^{1} \rightarrow S^{1}$ be a family of uniformly quasi-symmetric homeomorphisms of $S^{1}$, i.e. there exists a constant $M$ such that
\[
	\sup_{n}\sup_{Q}| \log |cr(\phi_{n}(Q))| | \leq M \ ,
\]
where the supremum is taken over all possible quadruples $Q$ of points in $S^{1}$ with $cr(Q)=-1$.
Then there exists a subsequence $\phi_{n_{k}}$ for which one of the following holds:
\begin{itemize}
	\item the homeomorphisms $\phi_{n_{k}}$ converge uniformly to a quasi-symmetric homeomorphism $\phi$;
	\item the homeomorphisms $\phi_{n_{k}}$ converge uniformly on the complement of any open neighborhood of a point of $S^{1}$ to a constant map.
\end{itemize}
\end{prop}

In terms of Anti-de Sitter geometry, the above proposition can be translated as follows:
\begin{prop}\label{prop:stabilityquasicircle} Let $\Gamma_{n}$ be a sequence of uniformly quasi-circles, i.e. $\Gamma_{n}$ are graphs of a family of unifomly quasi-symmetric homeomorphisms. Then, there exists a subsequence $\Gamma_{n_{k}}$ which converges in the Hausdorff topology either to the boundary of a light-like plane or to a quasi-circle $\Gamma_{\infty}$.
\end{prop}

Under particular assumptions, we can guarantee that the limit of a sequence of quasicircle is never the boundary of a light-like plane:
\begin{lemma}\label{lm:stabilityquasicircle}Let $S$ be a space-like surface whose asymptotic boundary is a quasi-circle $\Gamma$. Suppose that there exists $\epsilon \in (-\pi/2, \pi/2)$ such that the surface $S_{\epsilon}$ at time-like distance $\epsilon$ from $S$ is convex. Fix a point $x_{0}\in AdS_{3}$ and a future-directed unit vector $n_{0}\in T_{x_{0}}AdS_{3}$. Let $x_{n}$ be a sequence of points in $S$ and let $\phi_{n}$ be a sequence of isometries of $AdS_{3}$ such that $\phi_{n}(x_{n})=x_{0}$ and the future-directed normal vector to $S$ at $x_{n}$ is sent to $n_{0}$. Then no subsequences of $\phi_{n}(\Gamma)$ converge to the boundary of a light-like plane.
\end{lemma} 
\begin{proof} We can choose an affine chart such that $x_{0}=(0,0,0)\in \R^{3}$ and $n_{0}=(1,0,0)$. Let $x_{n}' \in S_{\epsilon}$ be the point corresponding to $x_{n}$ under the normal flow. In particular $d(x_{n}, x_{n}')=\epsilon$ for every $n$. Since $\phi_{n}$ is an isometry, we deduce that $\phi_{n}(x_{n}')=x_{0}'=(\arcsin(\epsilon),0,0)$. Moreover, if we denote with $P_{n}$ the totally geodesic space-like plane tangent to $S_{\epsilon}$ at $x_{n'}$, by construction $\phi_{n}(P_{n})=P_{0}$, where $P_{0}$ is the totally geodesic space-like plane through $x_{0}'$ orthogonal to $n_{0}$. Notice that, since $S_{\epsilon}$ is equidistant from $S$, they have the same asymptotic boundary $\Gamma$ and $S_{\epsilon}$ is contained in the domain of dependence of $\Gamma$ \\
\indent Suppose by contradiction that there exists a subsequence, still denoted with $\phi_{n}(\Gamma)$, which converges to the boundary of a light-like plane. Let $\xi\in \partial_{\infty}AdS_{3}$ be the self-intersecion point of the boundary at infinity of this light-like plane. Since $S_{\epsilon}$ is convex, $P_{n}$ is a support plane, hence its boundary at infinity is disjoint from $\Gamma$. After applying the isometry $\phi_{n}$, this implies that $P_{0}$ is disjoint from $\phi_{n}(\Gamma)$ for every $n$. We deduce that $\xi$ lies in $P_{0}$. But this implies that for $n$ big enough, the point $x_{0}'$ is not contained in the domain of dependence of $\phi_{n}(\Gamma)$, which is impossible because each $x_{n}'$ is contained in the domain of dependence of $\Gamma$ for every $n$.
\end{proof}

\indent We have now all the tools to study the principal curvature of an $H$-surface with bounded second fundamental form. In Section \ref{sec:existence}, we have seen that if we express the principal curvatures of an $H$-surface as $\pm \lambda+H$, and we define $\mu=\log(\lambda)$, then $\mu$ satisfies the differential equation
\begin{equation}\label{eq:ODEprincipalcurvature}
	\Delta \mu= e^{2\mu}-H^{2}-1  \ .
\end{equation}
\\
\indent The main result of this section is the following:

\begin{prop}\label{prop:boundGausscurvature} Let $S$ be an $H$-surface with bounded principal curvatures. If its boundary at infinity is a quasi-circle then $S$ is uniformly negatively curved.
\end{prop}

We will first show that an $H$-surface bounding a quasi-circle cannot be flat. In case $H=0$, a flat maximal surface in $AdS_{3}$ was described in \cite{maxsurface} as a maximal horosphere in $AdS_{3}$. It turns out that for general $H$, flat constant mean curvature surfaces are equidistant surfaces from this maximal horosphere. \\

\indent Let us first recall the construction of the maximal horosphere. Consider a space-like line $l$ in $AdS_{3}$ and its dual line $l^{\perp}$, which is obtained as the intersection of totally geodesic planes dual to points of $l$. We recall that $l^{\perp}$ can also be described as the set of points at distance $\pi/2$ from $l$. The maximal horosphere $S_{0}$ is defined as the  set of points at distance $\pi/4$ from $l$. It can be easily checked that $S_{0}$ is flat. Moreover, for every point $x\in S_{0}$, the surface $S_{0}$ has an orientation-reversing and time-reversing isometry obtained by reflection along a plane $P$ tangent to a point $x \in S_{0}$, followed by a rotation of angle $\pi/2$ around the time-like geodesic orthogonal to $P$ at $x$. This shows that the principal curvatures of $S_{0}$ must be opposite to each other, hence $S_{0}$ is a maximal surface. We then deduce by the Gauss formula that the principal curvatures are necessarily $\pm1$ at every point. We notice that the boundary at infinity of $S_{0}$ consists of four light-like segments, hence $S_{0}$ does not bound a quasi-circle.\\
\\
\indent We are now going to prove that flat $H$-surfaces can be obtained as surfaces at constant distance from the maximal horosphere. We will need the following well-known formulas for the variation of the induced metric and the shape operator in a foliation by equidistant surfaces.
\begin{prop}[Lemma 1.14 in \cite{Seppimaxsurface}]\label{prop:formulasequidistantsurfaces} Let $S$ be a space-like surface in $AdS_{3}$ with induced metric $I$ and shape operator $B$. Let $S_{\rho}$ be the surface at time-like distance $\rho$ from $S$, obtained by following the normal flow. Then the induced metric on the surface $S_{\rho}$ is given by 
\[
	I_{\rho}=I((\cos(\rho)E+\sin(\rho)B)\cdot, (\cos(\rho)E+\sin(\rho)B)\cdot) \ .
\]
Moreover, the shape operator of $S_{\rho}$ is given by
\[
	B_{\rho}=(\cos(\rho)E+\sin(\rho)B)^{-1}(-\sin(\rho)E+\cos(\rho)B) \ .
\]
\end{prop}

If we apply the previous proposition to the maximal horosphere $S_{0}$, we obtain that, choosing a local orthogonal frame that diagonalises $B$, the induced metrics and the shape operator of the surface at time-like distance $\rho$ from $S_{0}$ can be written as
\vspace{-0.5cm}
\begin{multicols}{2}
\[
	\ \ \ \ \ I_{\rho}=\begin{vmatrix}
				\cos(\rho)+\sin(\rho) & 0\\
				  \   &  \  \\
				0 & \cos(\rho)-\sin(\rho)\\
			\end{vmatrix} \ \ \ \ \ \ 
	B_{\rho}=\begin{vmatrix}
				\frac{-\sin(\rho)+\cos(\rho)}{\cos(\rho)+\sin(\rho)} & 0\\
				\ & \ \\
				0 & \frac{-\sin(\rho)-\cos(\rho)}{\cos(\rho)-\sin(\rho)}\\
			\end{vmatrix} \ .
\]
\end{multicols}
We deduce that equidistance surfaces form $S_{0}$ are smooth for $\rho\in (-\pi/4, \pi/4)$ and, for every value of $\rho$ in this interval, the surface $S_{\rho}$ is flat. Namely, by the Gauss formula, 
\[
	\K_{S_{\rho}}=-1-\det(B_{\rho})=-1+1=0 \ .
\]
Moreover, since the shape operator $B_{\rho}$ has constant trace
\[
	\trace(B_{\rho})=\frac{-\sin(\rho)+\cos(\rho)}{\cos(\rho)+\sin(\rho)}+\frac{-\sin(\rho)-\cos(\rho)}{\cos(\rho)-\sin(\rho)}=-2\tan(2\rho)
\]
we deduce that a flat surface with constant mean curvature $H$ is the surface at time-like distance $\rho=\frac{1}{2}\arctan(-H)$ from the maximal horosphere. Notice that, since these surfaces are at bounded distance from the maximal horosphere, they have the same boundary at infinity, which, we recall, consists of four light-like segments.\\
\\
\indent The proof of Proposition \ref{prop:boundGausscurvature} will then follow from the above description of flat $H$-surfaces by applying the Maximum Principle "at infinity":

\begin{proof}[Proof of Proposition \ref{prop:boundGausscurvature}]  Since $S$ has bounded second fundamental form, its Gaussian curvature is bounded. We will denote with $\K_{\sup}$ the upper bound of the Gaussian curvature of $S$. \\
\indent If $\K_{\sup}$ is attained, then the maximum principle applied to Equation (\ref{eq:ODEprincipalcurvature}) implies that $\K_{\sup}\leq 0$ and if $\K_{\sup}=0$, then the surface is flat. This latter case cannot happen, since by hypothesis $S$ bounds a quasi-circle, hence, if the upper bound of the Gaussian curvature is attained, the surface $S$ is uniformly negatively curved. \\
\indent We will know apply the Maximum Principle "at infinity" to get the same conclusion when the upper bound is not attained. Consider a sequence of points $x_{n}\in S$ such that the Gaussian curvature of $S$ at $x_{n}$ satisfies
\[
	\K_{\sup}-\frac{1}{n} \leq \K(x_{n})\leq \K_{\sup} \ .
\]
Let $\phi_{n}$ be a sequence of isometries of $AdS_{3}$ which sends $x_{n}$ to a fixed point $x_{0}$ and the future-directed unit normal vector to $S$ at $x_{n}$ to a fixed vector $n_{0}\in T_{x_{0}}AdS_{3}$. Since $S$ has bounded second fundamental form, Lemma \ref{lm:stability} shows that $\phi_{n}(S)$ converges, up to subsequences, to an $H$-surface $S_{\infty}$ in a neighborhood of $x_{0}$. By construction, the curvature of $S_{\infty}$ has a local maximum at $x_{0}$ equal to $\K_{\sup}$, hence the maximum principle applied again to Equation (\ref{eq:ODEprincipalcurvature}) shows that $\K_{\sup}\leq 0$. \\
\indent We are left to show that $\K_{\sup}\neq 0$. For this we will need to use that the boundary at infinity of $S$ is a quasi-circle and the description of flat constant mean curvature surfaces. Suppose by contradiction that $\K_{\sup}=0$. Then the sequence $\phi_{n}(S)$ converges in a neighborhood of $x_{0}$ to the flat $H$-surface $S_{\infty}$ described above. Moreover, Lemma \ref{lm:stability} implies that $\phi_{n}(S)$ converges, up to subsequences, to $S_{\infty}$ uniformly on compact set. In particular, the boundary at infinity of $\phi_{n}(S)$ converges to the boundary at infinity of $S_{\infty}$, which consists of four light-like segments. But this is not possible by Proposition \ref{prop:stabilityquasicircle}.
\end{proof}

\begin{oss}\label{oss:estimateprincipalcurvature} Proposition \ref{prop:boundGausscurvature} and the Gauss formula imply that the principal curvatures of an $H$-surface with bounded second fundamental form and asymptotic boundary a quasi-circle can be written as $\pm \lambda +H$, where $\lambda \in [0, \sqrt{H^{2}+1}-\epsilon]$ for some $\epsilon>0$. 
\end{oss}

We conclude this section with the following result about the existence of a convex surface equidistant from a negatively curved $H$-surface. This is a crucial step, together with Lemma \ref{lm:stabilityquasicircle}, to prove the uniqueness of the constant mean curvature foliation of the domain of dependence of a quasi-circle.

\begin{prop}\label{prop:existenceconvexsurface} Let $S$ be a negatively curved $H$-surface with bounded second fundamental form. Then there exists a convex surface equidistant from $S$.
\end{prop}
\begin{proof} We will do the proof for $H\geq 0$, the other case being analogous. Let us write the principal curvatures of $S$ as $\mu=\lambda +H$ and $\mu'=-\lambda+H$. By Remark \ref{oss:estimateprincipalcurvature}, we know that $\lambda\in [0, \sqrt{1+H^{2}}-\epsilon]$ for some $\epsilon>0$. We can clearly suppose $\epsilon<1$. By Proposition \ref{prop:formulasequidistantsurfaces}, we know that the principal curvatures of the surface $S_{\rho}$ at time-like distance $\rho$ from $S$ can be expressed as
\[
	\mu_{\rho}=\frac{\mu-\tan(\rho)}{1+\mu\tan(\rho)}=\tan(\rho_{0}-\rho)  
\]
\[
	\mu_{\rho}'=\frac{\mu'-\tan(\rho)}{1+\mu'\tan(\rho)}=\tan(\rho_{1}-\rho)
\]
where $\rho_{0}=\arctan(\mu)$ and $\rho_{1}=\arctan(\mu')$. We will denote 
\[
	\alpha=\arctan\big(\sqrt{1+H^{2}}+H-\epsilon\big) \ \ \ \beta=\arctan(H) \ \ \ \gamma=\arctan\big(H-\sqrt{1+H^{2}}+\epsilon\big) \ .
\]
By definition, we have
\[
	0< \beta < \rho_{0} < \alpha < \frac{\pi}{2} \ \ \ \ \gamma <\rho_{1}<\beta < \frac{\pi}{2} \ .
\]
We deduce that for every $\rho\in (\alpha-\pi/2, 0]$ the surface $S_{\rho}$ is smooth, since the principal curvatures are non-degenerate. Moreover, for every $\rho\in (\alpha-\pi/2, \gamma)$, the surface $S_{\rho}$ has positive principal curvatures. Here, we should be careful that $\alpha-\pi/2 <\gamma$, but this can easily be verified to be true under the assumption $0<\epsilon\leq 1$. Namely, $\alpha-\pi/2< \gamma$ if and only if
\[
	\arctan\big(H-\sqrt{1+H^{2}}+\epsilon\big)>\arctan\big(\sqrt{1+H^{2}}+H-\epsilon\big)-\frac{\pi}{2}=\arctan\left(\frac{-1}{H+\sqrt{1+H^{2}}-\epsilon}\right)
\] 
and it is sufficient to verify that 
\[
	\epsilon+H-\sqrt{1+H^{2}}>\frac{-1}{H+\sqrt{1+H^{2}}-\epsilon} \ ,
\]
which is true, under the hypothesis that $0<\epsilon\leq 1$, since
\[
	-1<-(1-\epsilon)^{2}\leq -\epsilon^{2}+2\epsilon\sqrt{1+H^{2}}-1=(\epsilon+H-\sqrt{1+H^{2}})(H-\epsilon+\sqrt{1+H^{2}}) \ .
\]
\end{proof}

\section{Uniqueness of the CMC foliation}\label{sec:uniqueness}
In Section \ref{sec:existence} we have proved the existence of a foliation by constant mean curvature surfaces of the domain of dependence of a quasi-circle. In this section we prove the uniqueness of such a foliation. \\
\\
\indent In order to prove the uniqueness of the foliation provided by Theorem \ref{thm:existencefoliation}, the main idea consists of proving that a constant mean curvature surface with bounded second fundamental form and with boundary at infinity a quasi-circle must coincide with a leaf of the foliation. The main tool we will use is the Maximum Principle for constant mean curvature surfaces in Lorentzian manifold, which we are going to recall.

\begin{lemma}[Maximum Principle, Lemma 2.3 in \cite{CMCfoliation}]\label{lm:maxprinciple} Let $\Sigma$ and $\Sigma'$ be smooth space-like surfaces in a time-oriented Lorentzian manifold $M$. Assume that $\Sigma$ and $\Sigma'$ are tangent at some point $p$, and assume that $\Sigma'$ is contained in the future of $\Sigma$. Then, the mean curvature of $\Sigma'$ at $p$ is smaller or equal than the mean curvature of $\Sigma$ at $p$.
\end{lemma}

We have now all the instruments to prove the uniqueness of the foliation by constant mean curvature surfaces of the domain of dependence of a quasi-circle.
\begin{teo}\label{thm:uniqueness} Let $\Gamma$ be a quasi-circle and let $D(\Gamma)$ be its domain of dependence. Then there exists a unique foliation of $D(\Gamma)$ by constant mean curvature surfaces with bounded second fundamental form.
\end{teo}
\begin{proof} The existence is provided by Theorem \ref{thm:existencefoliation}. In particular, we can define a time function $F: D(\Gamma) \rightarrow \R$ with the property that the level sets $F^{-1}(H)$ are $H$-surfaces for each $H \in \R$. \\
\indent To prove the uniqueness of this foliation, we are going to show that every other surface $S$ with constant mean curvature $H$, with bounded second fundamental form and with asymptotic boundary $\Gamma$ must coincide with the level set $F^{-1}(H)$. \\
Consider the restriction of $F$ to $S$. Suppose that $F$ admits maximum $h_{max}$ at a point $x \in S$ and minimum $h_{min}$ at a point $y \in S$. Notice that $h_{max}<+\infty$ and $h_{min}>-\infty$, otherwise $S$ would touch the domain of dependence. By construction, the surfaces $F^{-1}(h_{max})=S_{max}$ and $F^{-1}(h_{min})=S_{min}$ are tangent to $S$ at $x$ and $y$, respectively. Moreover, $S_{max}$ is in the future of $S$ and $S_{min}$ is in the past of $S$. Hence, by the Maximum Principle we obtain that
\[
	h_{max}\leq H \leq h_{min} \ .
\]
Therefore, $h_{min}=h_{max}=H$ and $S$ coincides with a level set of the function $F$ as claimed. \\
\indent In the general case, $F$ does not admit maximum and minimum on $S$, but we can still apply a similar reasoning "at infinity". We define
\[
	h^{+}=\sup_{x \in S}F(x)  \ \ \ \ \text{and} \ \ \ \ h^{-}=\inf_{ x \in S}F(x) \ .
\]
Fix a point $x_{0}$ in $AdS_{3}$ and a future-directed unit normal vector $n_{0}$ at $x_{0}$. Let $x_{n}$ be a sequence of points in $S$ such that $F(x_{n})$ tends to $h^{+}$ for $n \to \infty$. Let $\phi_{n}$ be a sequence of isometries of $AdS_{3}$ such that $\phi_{n}(x_{n})=x_{0}$ and the surface $\phi_{n}(S)$ is orthogonal to $n_{0}$ at $x_{0}$. By Lemma \ref{lm:stability}, the sequence of surfaces $S_{n}=\phi_{n}(S)$ converges $C^{\infty}$ to an $H$-surface $S_{\infty}$ in a neighborhood of $x_{0}$. Let $\Gamma_{\infty}$ be the boundary at infinity of $S_{\infty}$. Since $S_{n}$ converges uniformly on compact sets to $S_{\infty}$, the curve $\Gamma_{\infty}$ is the limit of the sequence $\phi_{n}(\Gamma)$. By Lemma \ref{lm:stabilityquasicircle} and Proposition \ref{prop:existenceconvexsurface}, the limit curve $\Gamma_{\infty}$ is a quasi-circle, hence Theorem \ref{thm:existencefoliation} provides a foliation of the domain of dependence of $\Gamma_{\infty}$ by constant mean curvature surfaces. Let $F_{\infty}: D(\Gamma_{\infty}) \rightarrow \R$ be the time function such that its level sets are leaves of the foliation.  By construction, the function $F_{\infty}$ restricted to $S_{\infty}$ admits maximum at $x_{0}$, so, by the Maximum Principle, $h^{+} \leq H$. Repeating a similar procedure for $h_{-}$, we obtain the inequalities  
\[
	h^{+} \leq H \leq h^{-}
\]
from which we deduce that $h^{+}=h^{-}=H$ and that $S$ must coincide with the level set $F^{-1}(H)$.
\end{proof}

\section{Application}\label{sec:application}
In this section we will use the theory described in Section \ref{sec:extensions} in order to associate to every constant mean curvature surface with asymptotic boundary the graph of a quasi-symmetric homeomorphism $\phi$ and with bounded second fundamental form a quasi-conformal extension of $\phi$. \\
\\
In Section \ref{sec:extensions}, we have recalled that, given a negatively curved space-like surface embedded in $AdS_{3}$, we can construct two local diffeomorphisms $\Pi_{l,r}:S \rightarrow \h^{2}$. 

\begin{prop}Let $S$ be an $H$-surface with bounded second fundamental form. Suppose that the boundary at infinity of $S$ is a quasi-circle $\Gamma$. Then the mappings $\Pi_{l,r}: S \rightarrow \h^{2}$ are global diffeomorphisms and they extend the maps $\pi_{l,r}:\Gamma \rightarrow S^{1}$
\end{prop}
\begin{proof} We will do the proof for the map $\Pi_{l}$, the other case being analogous. By Proposition \ref{prop:existenceconvexsurface}, there exists a convex surface $S_{\epsilon}$ equidistant from $S$. We remark that $S$ and $S_{\epsilon}$ has the same boundary at infinity. Let $\Pi_{l,\epsilon}$ be the map defined in Section \ref{sec:extensions} associated to the surface $S_{\epsilon}$. By Lemma \ref{lm:criterioestensione} and Remark \ref{oss:criterioestensione}, the map $\Pi_{l,\epsilon}$ is a global diffeomorphism and extends $\pi_{l}$. \\
\indent We claim now that, if we denote with $\eta: S \rightarrow S_{\epsilon}$ the diffeomorphism induced by the normal flow, we have $\Pi_{l}=\Pi_{l,\epsilon}\circ \eta$. This is sufficient to conclude the proof, as the map $\eta$ can be extended to the identity on the asymptotic boundaries. We now prove the claim. Let $p \in S$. Up to isometry we can suppose that, in the affine chart $U_{3}=\{x_{3}\neq 0\}$, we have $p=(0,0,0)$ and the tangent plane to $S$ at $p$ is the space-like plane $P$ of equation $x=0$. In addition, we can suppose that the totally-geodesic space-like plane we fix in the definition of $\Pi_{l}$ (see Section \ref{sec:extensions}) is exactly $P$. With this assumption, we clearly have $\Pi_{l}(p)=p$. Moreover, $\eta(p)=(\epsilon, 0,0)$ and the tangent plane to $S_{\epsilon}$ at $\eta(p)$ has equation $x=\epsilon$. On the other hand, since the left foliation is parametrised by
\[
	(x, \cos(\theta)-x\sin(\theta), \sin(\theta)+x\cos(\theta))=(x, \sqrt{1+x^{2}}\cos(\theta+\alpha), \sqrt{1+x^{2}}\sin(\theta+\alpha)) \ , 
\]
where $\tan\alpha=x$, we have that 
\[
 \Pi_{l, \epsilon}(\epsilon, t\sqrt{1+\epsilon^{2}}\cos(\theta), t\sqrt{1+\epsilon^{2}}\sin(\theta))=\frac{1}{\sqrt{1+\epsilon^{2}}}(0,t\cos(\theta-\beta), t\sin(\theta-\beta)) \ ,
\] 
where $\tan\beta=\epsilon$. Therefore, $\Pi_{l, \epsilon}(\eta(p))=\Pi_{l, \epsilon}(\epsilon, 0,0)=(0,0,0)=p=\Pi_{l}(p)$, as claimed. 
\end{proof}

As a consequence the map $\Phi=\Pi_{r}\circ \Pi_{l}^{-1}$ is a global diffeomorphism and extends the quasi-symmetric map $\phi$, whose graph is the quasi-circle $\Gamma$. We now use the estimates on the principal curvatures proved in Section \ref{sec:principalcurvatures} to show that $\Phi$ is quasi-conformal. 

\begin{prop}\label{prop:Hextension} Let $S$ be an $H$-surface with bounded second fundamental form whose boundary at infinity is the graph of a quasi-symmetric homeomorphism $\phi$. Then the associated map $\Phi$ is a quasi-conformal extension of $\phi$.
\end{prop}
\begin{proof}It is sufficient to prove that the map $\Pi_{l}: S \rightarrow \h^{2}$ is quasi-conformal from the induced metric on $S$ to the hyperbolic metric on $\h^{2}$. As seen in Section \ref{sec:extensions}, the differential of $\Pi_{l}$ is $E+JB$, where $J$ is the complex structure on $S$ and $B$ is its shape operator. We thus need to bound the module of the complex dilatation of the map $A=(E+JB)^{t}(E+JB)$. In a suitable orthogonal frame for $S$, we can suppose that $B$ is diagonal, hence
\[
	E+JB=\begin{vmatrix}
	1 & \lambda-H\\
	\lambda+H & 1\\ 
\end{vmatrix} \ ,
\]
where we wrote the principal curvatures of $S$ as $\pm\lambda+H$ with $\lambda \in [0, \sqrt{1+H^{2}}-\epsilon]$.
Therefore,
\[
	A=\begin{vmatrix}
	1+(H-\lambda)^{2} & 2\lambda\\
	2\lambda & (\lambda+H)^{2}+1\\
\end{vmatrix}
\]
and its complex dilatation is
\[
	\mu=-\frac{\lambda(H+i)}{1+H^{2}} \ .
\]
Thus,
\[
	|\mu|^{2}=\frac{\lambda^{2}}{1+H^{2}}=1-\frac{1+H^{2}-\lambda^{2}}{1+H^{2}}<1
\]
as wanted. 
\end{proof}

It turns out that the quasi-conformal homeomorphism $\Phi$ is a landslide. We recall that an area-preserving homeomorphism $\Phi:\h^{2}\rightarrow \h^{2}$ is a $\theta$-landslide if it can be decomposed as
\[
	\Phi=f_{2}\circ f_{1}^{-1}
\]
where $f_{i}:\h^{2}\rightarrow \h^{2}$ are harmonic maps whose Hopf differentials satisfy
\[
	\Hopf(f_{1})=e^{2i\theta}\Hopf(f_{2}) \ .
\]
\begin{prop}Let $S_{H} \subset AdS_{3}$ be a space-like $H$-surface bounding a quasi-circle at infinity. Then the map $\Pi_{r}\circ \Pi_{l}^{-1}$ is a $\theta$-landslide, where 
\[
	\theta=-\arctan(H)+\frac{\pi}{2} \ .
\]
\end{prop}
\begin{proof} Since, by definition, $\Phi=\Pi_{r}\circ \Pi_{l}^{-1}$, it is sufficient to prove that $\Pi_{l,r}:(S_{H},I) \rightarrow \h^{2}$ are harmonic with 
\[
	\Hopf(\Pi_{l})=e^{2i\theta}\Hopf(\Pi_{r}) \ .
\]
In order to prove that $\Pi_{r}$ is harmonic, it is sufficient to prove that if we write
\[
	\Pi_{r}^{*}g_{\h^{2}}=I(\cdot, b\cdot) \ ,
\]
then the traceless part of $b$ is Codazzi for $I$. By definition of the map $\Pi_{r}$, we know that
\[
	\Pi_{r}^{*}g_{\h^{2}}=I((E-JB)\cdot, (E-JB)\cdot)=I(\cdot, (E-JB)^{*}(E-JB))
\]
where $(E-JB)^{*}$ denotes the adjoint for the metric $I$. Since $B$ is $I$-self-adjoint and $J$ is skew-symmetric for $I$, we deduce that $(E-JB)^{*}=(E+JB)$, thus $b=(E+JB)(E-JB)$. Let us now decompose the operator $B$ as $B=B_{0}+HE$, where $B_{0}$ is traceless. Then
\begin{align*}
	(E+JB)(E-JB)&=E+BJ-JB+B^{2}\\
		&=E+B_{0}J+HJ-JB_{0}-HJ+B_{0}^{2}+H^{2}E+2HB_{0}\\
		&=(1+H^{2})E+2(HE-J)B_{0}+B_{0}^{2}\\
		&=(1+\lambda^{2}+H^{2})E+2(HE-J)B_{0}
\end{align*}
where, in the last passage, we have used the fact that $B_{0}^{2}=\lambda^{2}E$, $\lambda$ being the positive eigenvalue of $B_{0}$. We deduce that the traceless part $b_{0}$ of $b$ is given by $b_{0}=2(HE-J)B_{0}$, which is Codazzi because $H$ is constant, $B$ is Codazzi and $J$ is integrable and compatible with the metric $I$ (hence $d^{\nabla^{I}}J=0$). Therefore, $I(\cdot, b_{0}\cdot)$ is the real part of a holomorphic quadratic differential $\Phi_{r}$ and the metric $\Pi_{r}^{*}g_{\h^{2}}$ can be written as
\[
	\Pi_{r}^{*}g_{\h^{2}}=I(\cdot, b\cdot)=e_{r}I+\Phi_{r}+\overline{\Phi_{r}}
\]
where $e_{r}:S\rightarrow \R^{+}$ is the energy density of the map $\Pi_{r}$. Since $\Phi_{r}$ is holomorphic for the complex structure $J$, this shows that $\Pi_{r}$ is harmonic with Hopf differential 
\[
	\Hopf(\Pi_{r})=\Phi_{r}=I(HE-J)B_{0}+iIJ(HE-J)B_{0} \ .
\]
A similar computation for $\Pi_{l}$ shows that
\[
	\Hopf(\Pi_{l})=\Phi_{l}=I(HE+J)B_{0}+iIJ(HE+J)B_{0} \ .
\]
By using conformal coordinates, we deduce that
\[
	\frac{\Hopf(\Pi_{l})}{\Hopf(\Pi_{r})}=\frac{H+i}{H-i}=e^{2i\theta}
\]
with 
\[
	\theta=-\arctan(H)+\frac{\pi}{2}
\]
as claimed. 
\end{proof}

\bigskip

\noindent \footnotesize \textsc{DEPARTMENT OF MATHEMATICS, UNIVERSITY OF LUXEMBOURG}\\
\emph{E-mail address:}  \verb|andrea.tamburelli@uni.lu|

\end{document}